\documentclass[12pt,a4paper]{amsart}
\usepackage{amsthm}
\usepackage{amsmath}
\usepackage{amsfonts}
\usepackage{amssymb}
\usepackage[left=2.8cm, right=2.8cm, top=3.0cm, bottom=3.0cm]{geometry}
\usepackage{color}
\usepackage[arrow, matrix, curve]{xy}
\usepackage[colorlinks=true,linkcolor=green]{hyperref}
\makeatletter
\renewcommand*{\eqref}[1]{
	\hyperref[{#1}]{\textup{\tagform@{\ref*{#1}}}}}
\makeatother
\newcommand{\ind}{{\rm ind}}

\newcommand{\conj}{\mathrm{conj}}
\newcommand{\cri}{\mathrm{cr}}

\newcommand{\beq}{\begin{equation}}
\newcommand{\beqn}{\begin{equation*}}

\newcommand{\conc}{\mathrm{con}}

\newcommand{\eeq}{\end{equation}}
\newcommand{\eeqn}{\end{equation*}}

\newcommand{\R}{\mathbb{R}}
\newcommand{\ric}{\textrm{Ric}}

\newcommand{\n}{\mathbb{N}}
\newcommand{\q}{\mathbb{Q}}

\theoremstyle{plain}
\newtheorem{theorem}{Theorem}

\newtheorem{lemma}{Lemma}

\theoremstyle{definition}

\theoremstyle{remark}
\newtheorem{remark}{Remark}
\newtheorem{number-env}[theorem]{}
\title{Upper bounds for the critical values
of homology classes of loops}
\author{Hans-Bert Rademacher}
\address{Mathematisches Institut, 
Universit{\"a}t Leipzig, D--04081 Leipzig, Germany}
\email{rademacher@math.uni-leipzig.de}
\urladdr{www.math.uni-leipzig.de/\symbol{126}rademacher}
\date{2022-03-26}
\subjclass[2020]{53C22, 58E10}
\keywords{closed geodesic, 
	critical value of a 
	homology class, free loop space,
	based loop space, Morse index,
	positive Ricci curvature, positive sectional
	curvature}
\begin{document}
\begin{abstract}
In this short note we discuss upper bounds for
the critical values of homology classes
in the based and free loop space of 
manifolds carrying a Riemannian or Finsler
metric of positive Ricci curvature.
In particular it follows that a shortest closed
geodesic on a
simply-connected $n$-dimensional
manifold of 
positive Ricci curvature $\ric \ge n-1$
has length $\le n \pi.$
\end{abstract}
\maketitle
\baselineskip 18pt
\section{Results}
We start with a compact differentiable
manifold $M$ equipped with a 
Riemannian metric $g$ resp. a
non-reversible Finsler metric $f.$
Then the corresponding \emph{norm}
$\|v\|$ of a tangent vector $v$ is defined
by $\|v\|^2=g(v,v)$ resp. $\|v\|=f(v).$
In the following we use as common notation
also for a Finsler metric the letter $g.$
For a non-negative integer
$k \in \mathbb{N}_0=\n \cup \{0\}$ let 
$\conj(k) \in (0,\infty]$ be the infimum of all
$L>0$ such that any geodesic 
$c$ of length at least $L$
has Morse index $
\ind_{\Omega}(c)$ at least $(k+1).$
Hence for any geodesic $c:[0,1]\longrightarrow M$
of length $>\conj (k)$ the Morse index
is at least $(k+1).$
By the \emph{Morse index theorem}
it follows that there are $(k+1)$ conjugate points $c(s)$
with $0<s<1,$ which are conjugate to $c(0)$ along
$c|[0,s].$ Here we count conjugate points with
multiplicity, cf.~\cite[Sec. 2.5]{Kl}.
And we conclude:
$\conj(k)\le (k+1)\,\conj(0).$
 
Let $P([0,1],M)$ be the space of $H^1$-curves
$\gamma:[0,1]\longrightarrow M$ on the manifold
$M.$
Let $l, E, F: P=P([0,1],M)\longrightarrow \R$
denote the following functionals on this space.
The \emph{length} $l(c),$ resp. the \emph{energy}
$E(c)$ is defined as
\begin{equation*}
l(\gamma)=\int_0^1 \|\gamma'(t))\|\,dt\,;\,
E(\gamma)=\frac{1}{2}\int_0^1 \|\gamma'(t)\|^2\,dt\,.
\end{equation*}
We use instead of $E$ the 
\emph{square root energy functional} $F: P([0,1],M)\longrightarrow\R$ with $F(\gamma)=\sqrt{2E(\gamma)},$
cf.~\cite[Sec. 1]{HR}.
For a curve parametrized proportional to arc length
we have $F(\gamma)=l(\gamma).$ We consider the following subspaces of $P.$
The \emph{free loop space} $\Lambda M$ is the subset of
loops $\gamma$ with $\gamma(0)=\gamma(1).$
For points $p,q\in M$ the space
$\Omega_{pq} M$ is the subspace of curves
$\gamma$ joining $p=\gamma(0)$ and
$q=\gamma(1).$ The \emph{(based) loop space}
$\Omega_p M$ equals $\Omega_{pp}M .$
As common notation
we use $X,$ i.e. $X$ denotes $\Lambda M, 
\Omega_{pq}M,$ or $ \Omega_pM.$
It is well known that the critical points of
the 
square root energy 
functional $F:X \longrightarrow \R$ are
geodesics joining $p$ and $q$ for $X=\Omega_{pq}(M),$
the closed (periodic) geodesics for $X=\Lambda M,$
and the geodesic loops for $X=\Omega_p(M).$
The index form
$I_c$ can be identified with the hessian 
$d^2 E(c)$ of the energy functional, for the two
cases $X=\Lambda M$ resp. $X=\Omega_{pq}M$
(allowing also $p=q$)
we obtain different indices
$\ind_{\Lambda} (c)$ resp.
$\ind_{\Omega} (c).$ 
If 
$c\in \Lambda M$ is a closed geodesic 
with index $\ind_{\Lambda}(c)$ then
for $p=c(0)$
it is at the same time a geodesic loop
$c \in \Omega_p M$
with index $\ind_{\Omega}(c).$
The difference 
$\conc (c)=\ind_{\Lambda}c -\ind_{\Omega} c$
is called \emph{concavity}. It satisfies 
$0 \le \conc (c) \le n-1,$
cf.~\cite[thm. 2.5.12]{Kl} for the Riemannian case
and~\cite[Sec. 6]{Ra04} for the Finsler case. 

We use the following notation for sublevel sets
of $F:$ 
$
X^{\le a}=\{\gamma \in X,; F(\gamma)\le a\},
X^a=\{\gamma \in X;F(\gamma)=a\}.
$
For a non-trivial homology class $h \in H_j(X,X^{\le b};R)$ we
denote by $\cri_X(h)$ the \emph{critical value}, i.e.
the minimal value $a\ge b$ such that
$h$ lies in the image of the homomorphism
$	H_j(X^{\le a},X^{\le b};R)
	\longrightarrow
	H_j(X,X^{\le b};R)$
induced by the inclusion, cf.~\cite[Sec.1]{HR}.
It follows that for a non-trivial homology
class $h \in H_j(X,X^{\le b};R)$ 
there exists a geodesic in $X$ with
length $l(c)=\cri_X(h).$ 
Its index satisfies $\ind_X(c)\le j.$

The Morse theory of the functional
$F: X \longrightarrow \R$ implies
\begin{theorem}
	\label{thm:one}
Let $M$ be a compact manifold endowed with a Riemannian
metric resp. non-reversible Finsler metric $g.$
Let $h \in H_*(X,X^{\le b};R)$ 
be a non-trivial homology class
of degree $\deg(h)$ 
for some coefficient field $R.$ Then $\cri_X(h)\le \conj(\deg(h))\le
(1+\deg(h))\,\conj(0),$ 
and
the homomorphism
\begin{equation*}
H_j(X^{\le \conj(\deg(h))},X^{\le b};R)\longrightarrow 
H_j(X,X^{\le b};R)
\end{equation*}
induced by the inclusion is surjective
for all $j\le \deg(h).$
\end{theorem}
For positive Ricci curvature 
$\ric$ and 
for positive sectional curvature $K$
(resp. positive flag curvature
$K$ in the case of
a Finsler metric) we obtain
in Lemma~\ref{lem:conj}  upper bounds for 
the sequence  $\conj(k), k\in \n_0.$ 
As a consequence we
obtain:
\begin{theorem}
	\label{thm:two}
	Let $(M,g)$ be a compact 
	$n$-dimensional
	Riemannian or Finsler
	manifold.
		
	\smallskip
	
	(a) If $\ric \ge (n-1) \delta$ 
	for $\delta >0$ then
$\cri_X (h) \le \pi (\deg (h)+1) /\sqrt{\delta}$
for a non-trivial homology class
$h \in H_{*}(X,X^{\le b};R)$
of degree $\deg(h).$

\smallskip

(b) If $K \ge \delta$ for $\delta >0$ then 
$\cri_X(h) \le \pi \{1+\deg (h)/(n-1)\} /\sqrt{\delta}$
for a non-trivial homology class
$
h \in H_{*}(X, X^{\le b};R)$
of degree $\deg (h).$

\smallskip

(c) If $K \le 1$ then 
$\cri_{\Omega}(h)\ge  \left[\deg(h)/(n-1)\right]\pi$
for $h \in H_*(\Omega_{pq}M;R)$ 
and $\cri_{\Lambda}(h)\ge \left\{\left[\deg(h)/(n-1)\right]-1\right\}\pi$
for $h \in H_*(\Lambda M, \Lambda^{\le b}M;R).$
Here for a real number $x$ we denote
by $[x]$ the largest integer $\le x.$
\end{theorem} 
As consequence from
Theorem~\ref{thm:two}(a) we obtain 
an upper bound for the length of
a shortest closed geodesic on a manifold of 
positive Ricci curvature:
\begin{theorem}
	\label{thm:three}
Let $(M,g)$ be a compact and simply-connected
Riemannian or
Finsler manifold
of dimension $n$ 
of positive Ricci curvature
$\ric\ge (n-1)\delta$ for some 
$\delta>0.$
And let 
$m$ be the smallest integer with
$1\le m\le n-1$ for which $M$ is
$m$-connected and $\pi_{m+1}(M)\not=0.$
We denote by $L=L(M,g)$ the length
of a (non-trivial) shortest closed geodesic.
Then 
$L \le \pi (m+1)/\sqrt{\delta},$
in particular $L \le \pi n /\sqrt{\delta}.$ 
\end{theorem} 
\begin{remark}
	(a) This improves the estimate
	$L \le 8\pi m\le 8 \pi (n-1)$ given
	in \cite[Thm. 1.2]{Ro}.
	
	\smallskip
	
	(b) If $(M,g)$ is not simply-connected and
	$\ric \ge (n-1)\delta$ for some positive
	$\delta$ then there is a shortest closed curve
	$c$
	which is homotopically non-trivial. This closed
	curve is a closed geodesic and
	$\ind_{\Lambda}(c)=\ind_{\Omega}(c)=0.$ From Lemma~\ref{lem:conj}
	we obtain $L(c)\le \pi /\sqrt{\delta}.$
	On the other hand choose $k\in \n$ such
	that $l(c^k)=kl(c)>\pi /\sqrt{\delta,}$
	here $c^k(t)=c(kt)$ denotes the $k$-th
	iterate of the closed geodesic $c.$
	Then we conclude from
	Remark~\ref{rem:morse-schoenberg}(a)
	that
	$\ind_{\Lambda}(c^k)\ge
	\ind_{\Omega}(c)\ge 1,$
	hence the closed geodesic $c$
	is not \emph{hyperbolic,}
	cf. \cite[Thm. 3.3.9]{Kl}.
	
	\smallskip
	
	(c) For a compact 
and simply-connected
Riemannian manifold $(M,g)$ of
positive sectional curvature $K\ge \delta$
it follows from the estimate
$\conj(n-1) \le 2\pi/\sqrt{\delta}$ that the length $L$ of a 
shortest closed geodesic satisfies
	$L\le 2\pi/\sqrt{\delta}.$ In the limiting case
	$L=2\pi/\sqrt{\delta}$ the metric is of constant
	sectional curvature, 
	cf.~\cite[Cor. 1]{Ra21}.	
\end{remark} 
\begin{theorem}
\label{thm:four}
Let $(M,g)$ be a compact Riemannian or
Finsler manifold
of dimension $n$ with $\ric \ge (n-1)\delta$
(resp. $K \ge \delta$) for some positive $\delta.$ 
For any pair $p,q\in M$ of points (also allowing
$p=q$) and $k \in \n$
there exist at least $k$ geodesics joining 
$p$ and $q$ (i.e. geodesic loops for $p=q$) with length
$\le (2(n-1)k+1)\pi/\sqrt{\delta},$	
(resp.
$\le (2k+1)\pi/\sqrt{\delta}$).
\end{theorem} 
\begin{remark}
	(a)
	This result improves the bounds
	$16\pi(n-1)k$ resp.
	$(16(n-1)k+1)\pi$ given in
	\cite[Thm. 1.3]{Ro} for $\delta=1.$
	
	\smallskip
	
	(b)
	Here two geodesics $c_1,c_2 \in \Omega_{pq}M$
	are called \emph{distinct} if 
	their lengths $l(c_1)\not=l(c_2)$ are distinct.
	From a geometric point of view this is not
	very satisfactory. 
	If we choose distinct points $p,q\in S^n$
	on the sphere with the standard metric of
	constant sectional curvature $K=1,$
	which are not antipodal points,
	then any geodesic joining $p$ and $q$ is part
	of the unique great circle 
	$c:\R \longrightarrow S^n$ through $p$ and $q.$
	So in this case the geodesics whose existence
	is claimed in Theorem~\ref{thm:four} all
	come from a single closed geodesic,
	cf. \cite[p.181]{Kl}.
	Closed geodesics are called
	\emph{geometrically distinct} if they are
	different as subsets of $M$ (or in the case
	of a non-reversible Finsler metric if 
	their
	orientations are different when they agree
	as subsets of $M$).
	If the metric $g$ is bumpy then there
	are only finitely many geometrically distinct closed
	geodesics below a fixed length. Hence for a bumpy metric for almost
	all pairs of points $p,q$ on $M$ there is no 
	closed geodesic through these points.
	Hence in this case the geodesics 
	constructed in Theorem~\ref{thm:four}
	do not come from a single closed
	geodesic.
	
	\smallskip
	
	(c) There are related 
	\emph{curvaturefree} estimates depending only on
	the diameter due to
	Nabutovsky and Rotman.
	In ~\cite{NR} they show that for
	any pair $p,q$ of points in a compact
	$n$-dimensional Riemannian manifold with diameter
	$d$ and for every $k \in \n$ there are
	at least
	$k$ distinct geodesics joining $p$ and $q$ 
	of length $\le 4nk^2d.$
	\end{remark}
 \section{Proofs}
 \begin{proof}[Proof of Theorem~\ref{thm:one}]
 	(a) We first give the proof for the case
 	$X=\Omega_{pq}M$ for points $p,q$
 	and for a homology class
 	$h \in H_k(\Omega_{pq}M, \Omega^{\le b}M;R).$
 	Here we also allow 
 	the case $p=q.$ 
 	We denote by $d: M \times M \longrightarrow \R$ the \emph{distance} induced by the metric $g.$
 	We choose a sequence
 	$(q_j)_{j\ge 1}\subset M$ such that $\lim_{j\to \infty} d(p,q_j)=0$
 	and such that along any geodesic joining $p$ and
 	$q_j$ the point $q_j$ is not a conjugate point
 	to $p.$
 	This is possible as a consequence of 
 	Sard's theorem, cf.
 	\cite[Cor. 18.2]{Mi} for the Riemannian
 	case and
 	\cite[Cor. 8.3]{Ra04} for the Finsler case.
 	 As a consequence the square root energy functional
 	$F_j=F:
 	\Omega_{pq_j}M \longrightarrow \R$ is a \emph{Morse function.}
 	There is a homotopy equivalence 
 	$\zeta_{qq_j}:
 	\Omega_{pq}M \longrightarrow \Omega_{pq_j}M$
 	between loops spaces with
 	$F(\gamma)=\lim_{j\to \infty}F(\zeta_{qq_j}(\gamma))$
 	for all $\gamma \in \Omega_{pq}M,$
 	cf. \cite[Lem.1]{Ra21}.
 	Let $h\in H_k(\Omega_{pq}M,\Omega_{pq}^{\le b}M;R),$
 	then it follows from Morse theory for the
 	functional $F_j$ that there is a 
 	geodesic $c_j$ 
 	joining $p$ and $q_j$
 	whose length $l(c_j)$ equals the
 	critical value $\cri (\zeta_{qq_j}(h))$
 	of the homology class
 	$\zeta_{qq_j,*}(h)\in H_k
 	(\Omega_{pq_j}M,\Omega_{pq_j}^{\le b}M;R).$
 	The Morse index
 	$\ind_{\Omega} (c_j)$ as critical point of $F_j$ equals
 	the degree of the homology class by the 
 	Morse lemma,
 	cf. \cite[Sec. 8]{Ra04}. By definition of $\conj(k)$ we
 	obtain $l(c_j)\le \conj(k).$ 
 	But since $\cri_{\Omega} (h)
 	=\lim_{j\to \infty} l(c_j)\le \conj (k)$ we
 	finally arrive at the claim
 	$\cri_{\Omega}(h)\le \conj(k).$
 	
 	\smallskip
 	
 	(b) Now we assume $X=\Lambda M.$
 	Then we use a sequence $g_j$ of bumpy Riemannian
 	or Finsler
 	metrics
 	converging to the metric $g$ 
 	with respect to the strong $C^r$ topology for
 	$r \ge 2,$ resp. $r\ge 4$ in the Finsler case.
 	 	We can choose such a sequence  	by the 
 	\emph{bumpy metrics theorem} for Riemannian metrics
 	due to Abraham~\cite{Ab} and Anosov~\cite{An}, and by the
 	generalization to the Finsler case, 
 	cf. \cite{RT2020}.
 	The square root energy functional
 	$F_j: \Lambda M \longrightarrow \R$ is then a
 	\emph{Morse-Bott function,} the critical set equals the
 	set of closed geodesics which is the union
 	of disjoint 
 	and non-degenerate critical
 	$S^1$-orbits. Hence all closed geodesics
 	are non-degenerate, i.e. there is no periodic Jacobi
 	field orthogonal to the geodesic. Then for any 
 	$j$ there is a closed geodesic of $g_j$
 	such that the length $l(c_j)$ with respect to $g_j$
 	equals the critical value
 	$\cri_{\Lambda,j}(h)$ with respect to $g_j.$
 	Hence Morse theory implies that
 	the index $\ind_{\Lambda}(c_j)\in \{k,k-1\},$
 	since the critical submanifold is $1$-dimensional.
 	Then $\ind_{\Omega,j}(c_j)\le k$ which implies
 	that the length $l_j(c_j)$ of $c_j$
 	with respect to the metric $g_j$ satisfies
 	$\cri_{\Lambda,j}(h)=l_j(c_j)\le \conj_j(k).$
 	Here $\conj_j(k)$ is defined with respect
 	to the metric $g_j.$
 	Then $\cri_{\Lambda}(h)=
 	\lim_{j\to \infty} \cri_{\Lambda,j}(h)
 	\le \lim_{j\to \infty}\conj_j(k)
 	=\conj(k).$
 	 \end{proof}
\begin{remark}	
	\label{rem:morse-schoenberg}
The Morse-Schoenberg
	comparison result~\cite[Thm. 2.6.2]{Kl},
	\cite[Lem.~3]{Ra04}
	implies: 
	Let $c:[0,1]\longrightarrow M$ be a
	geodesic of length $l(c),$ and
	$k \in \n.$
	
\smallskip

(a) If $\ric \ge (n-1)\delta$ 
for $\delta >0$
and if
$l(c) > \pi k /\sqrt{\delta},$
then $\ind_{\Omega}(c)\ge k.$	

\smallskip

(b) If $K \ge \delta$ 
for a positive $\delta$
and if 
$l(c) > \pi k/\sqrt{\delta},$
then $\ind_{\Omega}(c)\ge k(n-1).$

\smallskip

(c) If $K\le 1$ and if $l(c)\le \pi k$
then $\ind_{\Omega}(c)\le (k-1)(n-1).$ 
\end{remark}
This implies
\begin{lemma}
	\label{lem:conj}
	Let $(M,g)$ be a manifold with Riemannian
	metric resp. Finsler metric $g.$
	
	\smallskip
		
	(a) If $\ric \ge (n-1)\delta $
	for $\delta >0$ then
	$\conj(k)\le (k+1)\pi/\sqrt{\delta}$
	for $k \in \n_0.$	
	
	\smallskip
	
	(b) If $K \ge \delta$ for $\delta>0$
	we have
	$\conj(k(n-1))\le (k+1)\pi/\sqrt{\delta}$
	for $k\in \n_0.$
\end{lemma}
\begin{proof}[Proof of Theorem~\ref{thm:two}]
From Theorem~\ref{thm:one} and 
Lemma~\ref{lem:conj} we immediately obtain
the statements (a) and (b).
Statement (c) follows analogously to the arguments
in the proof of Theorem~\ref{thm:one} together
with Remark~\ref{rem:morse-schoenberg}(c) and
the estimate $\conc(c)\le n-1.$
\end{proof}
\begin{proof}[Proof of Theorem~\ref{thm:three}]
By assumption there is a homotopically non-trivial
map $\phi: S^{m+1}\longrightarrow M,$
which also defines a 
homotopically non-trivial map
$\tilde{\phi}: (D^{m},S^{m-1})\longrightarrow 
(\Lambda M, \Lambda^0 M),$
cf. \cite[Thm.~2.4.20]{Kl}.
This defines a non-trivial
homology class $h \in H_{m}(\Lambda M, \Lambda^0 M;R)$
for some 
coefficient field $R.$
Then there exists a closed geodesic
$c$ with length
$l(c)=\cri_{\Lambda}(h).$
We conclude from Theorem~\ref{thm:two}(a)
that $l(c)=
\cri_{\Lambda}(h)\le \pi (m+1)/\sqrt{\delta}.$ 
\end{proof}
\begin{proof}[Proof of Theorem~\ref{thm:four}]
Since $M$ is simply-connected 
we conclude from a minimal model for the
rational homotopy type of $\Omega M:$
There exists
a non-trivial cohomology class
$\omega \in H^{2l}(\Omega_{pq};\q)$
of even degree $2l$ 
for some $1\le l\le n-1,$
which is not a torsion class with respect to
the cup product, i.e. $\omega^k\not=0$ for all
$k\ge 1.$
There is a sequence $h_k\in H_*(\Omega_{pq}M,\Omega_{pq}^{\le b}M;R), k\ge 1$
of non-trivial homology classes
with $h_k=\omega \cap h_{k+1},
\deg(h_k)=2lk, k\ge 1.$ 
Here $\cap$ denotes the \emph{cap product.}

Then we use the principle of 
\emph{subordinated homology classes,}
cf.~\cite[p.225--226]{BTZ} and conclude:
$\cri_{\Omega}(h_k)\le 
\cri_{\Omega}(h_{k+1})$ for all
$k\ge 1.$
Here 
equality only holds if there
are infinitely many distinct
geodesics in $\Omega_{pq}(M)$
of equal length $l(c)=\cri_{\Omega}(h_k)=
\cri_{\Omega}(h_{k+1}).$
Hence we can assume that
$\cri_{\Omega}(h_k)< 
\cri_{\Omega}(h_{k+1})$
and obtain a sequence
$c_k \in \Omega_{pq}M$ of geodesics
with $l(c_k)=\cri_{\Omega}(h_k).$
Since $\deg(h_k)=2lk\le 2(n-1)k$
we obtain the claim from
Theorem~\ref{thm:two}.
\end{proof}
\begin{remark}
(a) If $M$ is simply-connected and compact then
it was shown by Gromov~\cite[Thm. 7.3]{Gr} that there exist
positive constants $C_1=C_1(g),C_2=C_2(g)$ depending on 
the metric $g$ such that for all
homology classes $h \in H_*(\Lambda M;R)$
the following inequalities hold:
\begin{equation*}
C_1 \, \cri_{\Lambda}(h)
<
\deg (h)
<
C_2 \, \cri_{\Lambda}(h)\,.
\end{equation*}

\smallskip

(b)
If $M=S^n$ is a sphere of dimension $n\ge 3$
it is shown in \cite[Thm. 1.1]{HR} that 
there are positive numbers $\overline{\alpha}
=\overline{\alpha}(g),
\beta=\beta(g),$ depending on $g$ such that 
\begin{equation*}
	\overline{\alpha}\, \cri_{\Lambda}(h)-\beta
	<
	\deg (h)
	<
	\overline{\alpha} \,\cri_{\Lambda}(h)+\beta
\end{equation*}
holds for all $h \in H_*(\Lambda S^n).$
The number $\overline{\alpha}$ is called
\emph{global mean frequency.}
In case of positive Ricci curvature 
$\ric\ge (n-1)\delta$ we conclude
from Theorem~\ref{thm:two}(a):
$\sqrt{\delta}/\pi \le
\overline{\alpha}.$
If $K\le 1$ then $\overline{\alpha}\le 
(n-1)/\pi.$
\end{remark}


\begin{thebibliography}{999999}
\bibitem{Ab}
R. Abraham, Bumpy metrics, in: 
\emph{Global Analysis. Proc.Symp.Pure Math}
Vol. XIV Amer.Math.Soc. Providence R.I.(1970) 1–3
\bibitem{An} D.V. Anosov, On generic properties of closed geodesics,
\emph{Izv.Akad.Nauk. SSSR 46(1982)=
Math. USSR Izv. 21} (1983) 1–29
\bibitem{BTZ}
W.Ballmann, G.Thorbergsson, \& W.Ziller,
Existence of closed geodesics on
positively curved manifolds,
\emph{J.~Differential~Geom.} 18 (1983)
221--252
\bibitem{Gr}
M. Gromov, 
\emph{Metric structures for Riemannian 
and non-Riemannian spaces,}
Prog.~Math. 152. Birkh\"auser 
Boston, Inc., Boston, MA, 1999
\bibitem{HR} N.Hingston \& H.B.Rademacher,
Resonance for loop homology of spheres,
\emph{J.~Differential~Geom.} 93 (2013) 133--174 
\bibitem{Kl} W.Klingenberg,
\emph{Riemannian geometry.}
2nd edition, de Gruyter studies math. 1, de Gruyter, Berlin
New York 1995	
\bibitem{Mi} J. Milnor, \emph{Morse theory,} 
Annals~Math.~Studies 51 Princeton~Univ.
Press, Princeton 1969.
\bibitem{NR}
A.Nabutovsky \& R.Rotman,  
Length of geodesics and quantitative Morse theory on loop spaces. 
\emph{Geom.~Funct.~Anal.} 23 (2013) 367–414 
\bibitem{Ra04} H.B.Rademacher,
A sphere theorem for non-reversible Finsler metrics,
\emph{Math.~Ann.} 328 (2004) 373-387
\bibitem{Ra21}
H.B.Rademacher,
Critical values of homology classes of loops and
positive curvature,
\emph{J.~Differential~Geom.} 119 (2021) 141 - 158
\bibitem{RT2020} H.B.Rademacher \& I.A.Taimanov,
The second closed geodesic, the fundamental group, and generic Finsler metrics, 
{\tt arxiv2011.01909}
\bibitem{Ro} R.~Rotman,
Positive Ricci curvature and the length of a shortest
periodic geodesic, 
{\tt arxiv:2203.09492}
\end{thebibliography}
\end{document}